\documentclass[11pt]{article}
\usepackage{fullpage}

\date{}

\usepackage{tikz}
\usetikzlibrary{calc}
\usetikzlibrary{matrix}
\usepackage{amsmath,amssymb,amsfonts,amsthm,subfig}
\usepackage[pdftex]{hyperref} 

\usepackage{graphicx}
\allowdisplaybreaks

\newcommand{\V}[1]{\mbox{\boldmath $ #1 $}}

\def \p{\partial}

\newcommand{\bey}{\begin{eqnarray}}
\newcommand{\eey}{\end{eqnarray}}

\newcommand{\beq}{\begin{equation}}
\newcommand{\eeq}{\end{equation}}
\theoremstyle{plain}
\newtheorem{thm}{\hspace{6mm}Theorem}[section]
\newtheorem{pro}{\hspace{6mm}Proposition}[section]
\newtheorem{lem}{\hspace{6mm}Lemma}[section]

\theoremstyle{definition}

\theoremstyle{remark}

\newtheorem{rem}{\hspace{6mm}Remark}[section]

\def\nablat{\tilde{\nabla}}

\title{Mesh Sensitivity Analysis for Finite Element Solution of Linear Elliptic Partial Differential Equations}

\author{Yinnian He%
\thanks{School of Mathematics and Statistics, Xi'an Jiaotong University, Xi'an 710049, China
(heyn@mail.xjtu.edu.cn).}
\and Weizhang Huang%
\thanks{Department of Mathematics, University of Kansas, Lawrence, KS 66045, U.S.A. (whuang@ku.edu).}}

\begin{document}
\vskip 1cm
\maketitle

\begin{abstract}
Mesh sensitivity of finite element solution for linear elliptic partial differential equations
is analyzed. A bound for the change in the finite element solution is obtained in terms of the mesh deformation and its gradient.
The bound shows how the finite element solution changes continuously with the mesh.
The result holds in any dimension and for arbitrary unstructured simplicial meshes, general linear elliptic
partial differential equations, and general finite element approximations.
\end{abstract}



\noindent{\bf Abbreviated title.} Mesh Sensitivity Analysis for Finite Element Solution

\section{Introduction}

We are concerned with mesh sensitivity for the finite element (FE) solution of
linear elliptic boundary value problems (BVPs).
The finite element method is a well developed method that is widely used in scientific and engineering computation
and has been studied extensively in the numerical analysis community. In particular,  numerous error estimates have
been established for the FE solution of elliptic problems; see, e.g., \cite{AO00,BS94,Cia78}.
Those error estimates (cf. Proposition~\ref{pro:error-1} below) are typically established for a given type of mesh
and show convergence behavior of the error as the mesh is being refined.
They also show the {\it stable} dependence of the FE solution on the mesh in the sense that
the FE solution remains in a neighborhood of the exact solution for all meshes of same type no matter how different they are.
However, those error estimates do not tell if and how the FE solution changes continuously with the mesh.
Although it is commonly believed and numerically verified that the FE solution depends continuously
on the mesh at least for linear elliptic problems, little work has been done so far on the theoretical study of
this mesh sensitivity issue.
To our best knowledge, \cite{HH2016} is the only known work on this issue where a bound for the change
of the linear FE solution resulting from the mesh deformation has been obtained for
one-dimensional elliptic boundary value problems.

The objective of this work is to present an analysis of the mesh sensitivity of the FE solution
to a BVP of genera linear elliptic partial differential equations (PDEs) defined on a polygonal/polyhedral
domain in $\mathbb{R}^d$ ($d \ge 1$).
We consider an arbitrary unstructured simplicial mesh and a general order FE approximation.
We use an approach similar to gradient methods
for optimal control (see, e.g. \cite{BH1975}) and sensitivity analysis in shape optimization (e.g., see \cite{HM2003,SZ92})
where a small deformation in the mesh is introduced and then an FE formulation is derived and bounds are established
for the change in the FE solution resulting from the mesh deformation. The main results are stated in
Theorems~\ref{thm:fem-1} and \ref{thm:fem-2}.

An outline of the paper is as follows. The BVP under consideration and its FE formulation are described in Section~\ref{SEC:FE}.
In Section~\ref{SEC:mesh-deform}, the mesh deformation is introduced and changes in mesh quantities and functions
resulting from the mesh deformation are discussed. The mesh sensitivity of the FE solution is analyzed
in Section~\ref{SEC:sensitivity}. Numerical results are presented in Section~\ref{SEC:numerics} for an example with
smooth and nonsmooth velocity fields. Finally, conclusions are
drown in Section~\ref{SEC:conclusions}.

\section{Finite element formulation}
\label{SEC:FE}

We consider the boundary value problem of a general linear elliptic PDE as
\bey
&& - \nabla (a \nabla u) + \V{b}\cdot \nabla u + c u  = f,\quad \mbox{ in } \Omega
\label{bvp-1}
\\
&& u = 0,\qquad \mbox{ on } \p \Omega \label{bc-1}
\eey
where $\Omega$ is a polygonal/polyhedral domain in $\mathbb{R}^d$ ($d \ge 1$) and 
the coefficients $a(\V{x})$, $\V{b}(\V{x})$, $c(\V{x})$,
and $f(\V{x})$ are given functions satisfying
\begin{align*}
& a,\; \V{b},\; c \in W^{1,\infty}(\Omega), \quad f \in W^{1,2}(\Omega),
\\
& a(\V{x}) \ge a_0 > 0, \quad c(\V{x}) - \frac{1}{2} \nabla \cdot \V{b}(\V{x}) \ge 0, \quad \text{ in } \Omega . 
\end{align*}
The derivatives of the coefficients will be needed in the sensitivity analysis (cf. Section~\ref{SEC:sensitivity}).
The coefficients can have lower regularity if only the convergence of the FE solution
is concerned.

Let $V =H_0^1(\Omega)$. The variational formulation of (\ref{bvp-1}) and (\ref{bc-1}) is to find $u \in V$ such that
\beq
\int_{\Omega} \left ( a \nabla u \cdot \nabla \psi  + (\V{b}\cdot \nabla u)\psi  + c u \psi\right ) d \V{x}
= \int_{\Omega} f \psi d \V{x},
\quad \forall \psi \in V.
\label{bvp-2}
\eeq
It can be shown that 
\begin{align}
\int_{\Omega} \left ( a \nabla u \cdot \nabla u  + (\V{b}\cdot \nabla u) u  + c u^2\right ) d \V{x}
\ge a_0 \| \nabla u \|_{L^2(\Omega)}^2, \quad \forall u \in V .
\label{bvp-3}
\end{align}

We consider the FE solution of problem
(\ref{bvp-2}). To this end, we assume that a simplicial mesh $\mathcal{T}_h$ has been given for $\Omega$.
Let $K$ be the generic element of $\mathcal{T}_h$
and $h_K$ and $a_K$ be the diameter (defined as the length of the longest edge) and the minimum height of $K$,
respectively. Here, a height of $K$ is defined as the distance from a vertex to its opposite facet.
The mesh $\mathcal{T}_h$ is said to be regular if there exists a constant $\kappa > 0$ such that
\beq
\frac{h_K}{a_K}\le \kappa, \quad \forall K\in \mathcal{T}_h.
\label{mesh-reg-1}
\eeq
Although it is common to assume that the mesh is regular in FE error analysis (cf. Proposition~\ref{pro:error-1}),
we do not make such an assumption in our current mesh sensitivity analysis. Instead, we only require
$a_K > 0$ for all $K \in \mathcal{T}_h$, which essentially says that all elements must not be degenerate or inverted.

Edge matrices of mesh elements are a useful tool in our analysis.
Denote the vertices of $K$ by $\V{x}_i^K, i = 0, ..., d$. An edge matrix of $K$ is defined as
\[
E_K = [\V{x}_1^K-\V{x}_0^K, ...,\V{x}_d^K-\V{x}_0^K].
\]
It is evident that this definition is not unique, depending on the ordering of the vertices. Nevertheless, 
many geometric properties of $K$, which are independent of the ordering of the vertices, can be computed
using $E_K$. For example, the volume of $K$ can be calculated by $|K| = \det(E_K)/d!$.
Moreover, it is known \cite{LH2017} that
\begin{equation}
E_K^{-T} = [\nabla \phi_1^K, ..., \nabla \phi_d^K],
\label{E-basis-1}
\end{equation}
where $\phi_i^K$ is the linear Lagrange basis function associated with $\V{x}_i^K$. These basis functions satisfy
$\sum_{i=0}^d \phi_i^K = 1$. It is also known that the $i$-th height of $K$ is equal to $1/|\nabla \phi_i^K|$ and thus,
\begin{equation}
a_K = \min_{i} \frac{1}{|\nabla \phi_i^K|}.
\label{aK-1}
\end{equation}

We consider the FE space associated with $\mathcal{T}_h$ as
\[
V_h=\{v\in H^1_0(\Omega)\cap C(\overline{\Omega});\;  v|_K\in P_r(K),~~\forall K\in \mathcal{T}_h\},
\]
where $P_r(K)$ ($r \ge 0$) is the set of polynomials of degree no more than $r$ defined on $K$.
Any function $v_h$ in $V_h$ can be expressed as
\[
v_h = \sum_i v_i \psi_i(\V{x}),
\]
where $\{\psi_1, \psi_2, ...\}$ is a basis for $V_h$. We distinguish FE basis functions
$\{\psi_1, \psi_2, ...\}$ from the linear Lagrange basis functions $\{ \phi_i, i = 1, 2, ...\}$
(with $\phi_i$ being associated with $\V{x}_i$) and emphasize that they can be different
(even when $r = 1$).
The FE solution of BVP (\ref{bvp-1}) and (\ref{bc-1}) is to find
$u_h \in V_h$ such that
\beq
\int_{\Omega} \left ( a \nabla u_h \cdot \nabla \psi  + (\V{b}\cdot \nabla u_h)\psi  + c u_h \psi\right ) d \V{x}
 = \int_{\Omega} f \psi d \V{x}, \quad \forall \psi \in V_h.
\label{fem-1}
\eeq

The following proposition is a standard error estimate that can be found in most FEM textbooks (e.g., see \cite{BS94}).
\begin{pro}
\label{pro:error-1}
Assume that $u \in H^2(\Omega)$ and the mesh $\mathcal{T}_h$
is regular. Then,
\beq
\| \nabla (u_h - u)\|_{L^2(\Omega)} \le C h \|\nabla^2 u \|_{L^2(\Omega)} ,
\label{error-1-1}
\eeq
where $h = \max_{K \in \mathcal{T}_h} h_K$ and $C$ is a constant independent of $u$, $u_h$, and the mesh.
\end{pro}

The error estimate (\ref{error-1-1}) shows the stable dependence of the FE solution
on the mesh. It shows that the FE solution remains
in a neighborhood of the exact solution for all regular meshes with
maximum element diameter $h$ no matter how different they are.
However, the estimate does not tell if and how the FE solution changes continuously with the mesh.

\section{Mesh deformation}
\label{SEC:mesh-deform}

For the mesh sensitivity analysis of the FE solution we use an approach similar to gradient methods
for optimal control (see, e.g. \cite{BH1975}) and sensitivity analysis in shape optimization (e.g., see \cite{HM2003,SZ92}).
In this approach, a small deformation in the mesh is introduced and then an FE formulation is derived and bounds are
established for the change in the FE solution resulting from the mesh deformation.
In this section, we focus on the mesh deformation and changes in mesh qualities and functions
resulting from the mesh deformation.
We will discuss the mesh sensitivity of the FE solution in the next section.

We assume that a smooth vector field $\dot{X}=\dot{X}(\V{x})$ is given on $\Omega$ and satisfies
\[
\| \dot{X} \|_{L^\infty(\Omega)} < \infty, \quad \| \nabla \dot{X} \|_{L^\infty(\Omega)} < \infty .
\]
We consider the deformation of the mesh $\mathcal{T}_h$ by keeping its connectivity and moving
its interior vertices according to
\beq
\V{x}_i(t) = \V{x}_i(0) + t  \dot{\V{x}}_i, \quad 0 \le t < \delta, \quad i = 1, 2, ...
\label{xt-1}
\eeq
where $\delta$ is a small positive number and the nodal velocities are defined as
$\dot{\V{x}}_i =  \dot{X}(\V{x}_i(0))$ (that are considered constant in time). 
We denote the time-dependent mesh by $\mathcal{T}_h(t)$.
Here, we fix the boundary vertices for notational simplicity. 
The analysis applies without major modifications if the boundary points are allowed to move.
We also note that the linearization of any smooth mesh deformation can be cast in the form of (\ref{xt-1}).
Thus, (\ref{xt-1}) is sufficiently general.

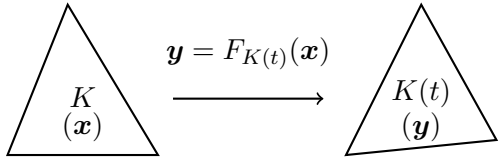
\begin{figure}[thb]
\centering
\begin{tikzpicture}[scale = 1]
\draw [thick] (0,0) -- (2,0) -- (0.8, 2) -- cycle;
\draw [thick,->] (2.2,0.75) -- (4.2,0.75);
\draw [above] (3.2,1) node {$\V{y} = F_{K(t)}(\V{x})$};
\draw [thick] (4.5,0) -- (6.5,0.2) -- (5.5, 2) -- cycle;
\draw [above] (1,0.5) node {$K$};
\draw [below] (1,0.7) node {$(\V{x})$};
\draw [above] (5.5,0.5) node {$K(t)$};
\draw [below] (5.5,0.7) node {$(\V{y})$};
\end{tikzpicture}
\caption{Affine mapping $F_{K(t)}$ from $K(0)=K$ to $K(t)$.}
\label{fig:mapping-1}
\end{figure}

To understand effects of the mesh deformation on mesh quantities and functions, we consider
the affine mapping $F_{K(t)}$ from $K(0) = K \in \mathcal{T}_h $ to $K(t) \in \mathcal{T}_h(t)$ (cf. Fig.~\ref{fig:mapping-1}).
To avoid notational confusion, we use coordinates $\V{x}$ and $\V{y}$ for $K(0)$ and $K(t)$, respectively.
Moreover, we express quantities and functions in $\V{y}$ by adding ``\,$\tilde{\mbox{ }}$\," on the top of the names.
For example, $\nabla$ denotes the gradient operator with respect to $\V{x}$ while the gradient operator
with respect to $\V{y}$ is written
as $\nablat$. We consider mesh deformation effects by first transforming functions/mesh quantities from $\V{y}$ to $\V{x}$ and
then differentiating them with respect to $t$ while keeping $\V{x}$ fixed. The time differentiation is similar to
material differentiation in fluid dynamics. We denote corresponding derivatives by the symbol ``\,$\dot{\mbox{ }}$\,".

We first consider time derivatives for the Jacobian matrix and determinant.
Denote the Jacobian matrix of $F_{K(t)}$ by $\mathbb{J} = \frac{\partial F_{K(t)}}{\partial \V{x}}$
and the Jacobian determinant by $J = \det(\mathbb{J})$. We can express $\mathbb{J}$ in terms of the edge matrices
of $K(0)$ and $K(t)$. Since $\V{y} = F_{K(t)}(\V{x})$ is affine, it can be expressed as
\beq
\V{y} = F_{K(t)}(\V{x}) = \V{x}_0^K(t) + \mathbb{J} \; (\V{x}-\V{x}_0^K(0)), \quad \forall \V{x} \in K(0).
\label{FK-1}
\eeq
By taking $\V{x} = \V{x}_i^K(0)$ and $\V{y} = \V{x}_i^K(t)$ ($i = 1, ..., d$) sequentially, we get
\[
[\V{x}_1^K(t) - \V{x}_0^K(t), ..., \V{x}_d^K(t) - \V{x}_0^K(t)]
= \mathbb{J}\; [\V{x}_1^K(0) - \V{x}_0^K(0), ..., \V{x}_d^K(0) - \V{x}_0^K(0)],
\]
which gives
\beq
\mathbb{J} = E_{K(t)} E_{K(0)}^{-1}.
\label{J-1}
\eeq
From this, it is evident that $\mathbb{J} |_{t = 0} = I$ (the $d\times d$ identity matrix).
Differentiating (\ref{FK-1}) with respect to $t$ gives
\beq
\dot{F}_{K(t)} = \dot{\V{x}}_0^K +  \dot{\mathbb{J}}\, (\V{x}-\V{x}_0^K(0))
= \dot{\V{x}}_0^K + \dot{E}_{K(t)} E_{K(0)}^{-1} (\V{x}-\V{x}_0^K(0)).
\label{FK-3}
\eeq
where we have used
\beq
\dot{\mathbb{J}} = \dot{E}_{K(t)} E_{K(0)}^{-1},
\quad \dot{E}_{K(t)} = [\dot{\V{x}}_1^K - \dot{\V{x}}_0^K, ..., \dot{\V{x}}_d^K - \dot{\V{x}}_0^K].
\label{Jdot-1}
\eeq

On the other hand, $\V{y} = F_{K(t)}(\V{x})$ can be expressed in terms of linear basis functions as
\beq
\label{FK-0}
F_{K(t)} = \sum_{i=0}^d \V{x}_i^K(t) \phi_i^K(\V{x}).
\eeq
Differentiating this with respect to $t$ yields
\beq
\dot{F}_{K(t)} = \sum_{i=0}^d \dot{\V{x}}_i^K \phi_i^K(\V{x}).
\label{FKdot-1}
\eeq

\begin{lem}
\label{lem:FKdot-0}
There hold
\begin{align}
& \dot{F}_{K(t)} = \dot{X}_h |_K,
\label{FKdot-2}
\\
& \nabla \cdot \dot{F}_{K(t)} = \text{tr} (\dot{\mathbb{J}}) = \text{tr} (\dot{E}_{K(t)} E_{K(0)}^{-1}) = \nabla \cdot \dot{X}_h |_K,
\label{FKdot-3}
\end{align}
where $\dot{X}_h$ is a piecewise linear velocity field defined as
\beq
\label{Xt-0}
\dot{X}_h = \sum_i \dot{\V{x}}_i \phi_i(\V{x}) .
\eeq
\end{lem}

\begin{proof}
The equality (\ref{FKdot-2}) follows from (\ref{FKdot-1}) and (\ref{Xt-0}).

Applying the divergence operator to (\ref{FK-3}) and by direct calculation, we get
\[
\nabla \cdot \dot{F}_{K(t)} = \text{tr} (\dot{\mathbb{J}}) = \text{tr} (\dot{E}_{K(t)} E_{K(0)}^{-1}) .
\]
Applying the divergence operator to (\ref{FKdot-1}), we get
\[
\nabla \cdot \dot{F}_{K(t)} = \nabla \cdot \dot{X}_h |_K .
\]
Combining the above results we obtain (\ref{FKdot-3}).
\end{proof}

\begin{lem}
\label{lem:Jdot-1}
There hold
\beq
\dot{\mathbb{J}}|_{t = 0}  = \dot{E}_{K(0)} E_{K(0)}^{-1},\quad
\dot{(\mathbb{J}^{-1})}|_{t = 0} = - \dot{E}_{K(0)} E_{K(0)}^{-1} ,\quad
\dot{J}|_{t=0} = \nabla \cdot \dot{X}_h|_K .
\label{der-1}
\eeq
\end{lem}

\begin{proof}
The first equation in (\ref{der-1}) follows from (\ref{Jdot-1}).

By differentiating the identity $\mathbb{J} \mathbb{J}^{-1} = I$, we get
\[
\dot{(\mathbb{J}^{-1})} = - \mathbb{J}^{-1} \dot{\mathbb{J}} \mathbb{J}^{-1} .
\]
Taking $t = 0$ and using the first equation in (\ref{der-1}) and the fact that $\mathbb{J}^{-1} |_{t = 0} = I$,
we obtain the second equation in (\ref{der-1}).

Recalling from $J = \det(\mathbb{J})$ and using the derivative formula for matrix determinants,  we get
\[
\dot{J} = \det(\mathbb{J}) \text{tr} (\mathbb{J}^{-1} \dot{\mathbb{J}}).
\]
Taking $t=0$ and using (\ref{FKdot-3}) we obtain the third equation in (\ref{der-1}).
\end{proof}

\begin{lem}
\label{lem:Xh-1}
There hold
\begin{align}
& \| \dot{X}_h \|_{L^\infty(\Omega)} \le \| \dot{X} \|_{L^\infty(\Omega)} ,
\label{Xdot-1}
\\
& \| \nabla \cdot \dot{X}_h\|_{L^\infty(\Omega)} \le d \| \nabla \dot{X} \|_{L^\infty(\Omega)} \max_K \frac{h_K}{a_K} .
\label{Xdot-2}
\end{align}
\end{lem}

\begin{proof}
The inequality (\ref{Xdot-1}) is evident from the definition of $\dot{X}_h$ in (\ref{Xt-0}).

On $K$, we can write $\nabla \cdot \dot{X}_h$ as
\beq
\nabla \cdot \dot{X}_h = \sum_{i=0}^d \dot{\V{x}}_i^K \cdot \nabla \phi_i^K(\V{x})
= \sum_{i=1}^d (\dot{\V{x}}_i^K-\dot{\V{x}}_0^K) \cdot \nabla \phi_i^K(\V{x}) .
\label{Xh-3}
\eeq
From
\[
\dot{\V{x}}_i^K-\dot{\V{x}}_0^K = \dot{X}(\V{x}_i^K)-\dot{X}(\V{x}_0^K)
= \int_0^1 \nabla \dot{X}(\V{x}_0^K + t (\V{x}_i^K-\V{x}_0^K)) \cdot (\V{x}_i^K-\V{x}_0^K) d t,
\]
we have
\beq
| \dot{\V{x}}_i^K-\dot{\V{x}}_0^K | \le \| \nabla \dot{X} \|_{L^\infty(\Omega)} h_K .
\label{Xdot-3}
\eeq
The inequality (\ref{Xdot-2}) follows from (\ref{aK-1}), (\ref{Xh-3}), and the above inequality.
\end{proof}

\begin{rem}
\label{rem:Xdot-1}
In (\ref{Xdot-2}) we have assumed that the mesh velocity field is smooth. If the mesh velocity field is not smooth,
from (\ref{Xh-3}) we have
\beq
\| \nabla \cdot \dot{X}_h\|_{L^\infty(\Omega)} \le \frac{d+1}{\min\limits_K a_K} \| \dot{X} \|_{L^\infty(\Omega)} .
\label{Xdot-4}
\eeq
\end{rem}

\begin{lem}
\label{lem:EK-1}
There hold
\begin{align}
& \|E_{K}^{-1}\|_2 \le \frac{\sqrt{d}}{a_K},
\label{EK-3}
\\
& \| \dot{E}_{K}\|_2 \le \sqrt{d} \, h_K \| \nabla \dot{X} \|_{L^\infty(\Omega)} ,
\label{EK-4}
\\
& \| \dot{E}_{K}\|_2 \|E_{K}^{-1}\|_2 \le \frac{d\, h_K}{a_K} \| \nabla \dot{X} \|_{L^\infty(\Omega)} ,
\label{EK-2}
\end{align}
where $\| \cdot \|_2$ denotes the 2-norm for matrices.
\end{lem}

\begin{proof}
From (\ref{E-basis-1}), we have
\[
\|E_{K(0)}^{-1}\|_2^2 \le \|E_{K(0)}^{-1}\|_F^2 = \sum_{i=1}^d |\nabla \phi_i^K|^2
\le \frac{d}{a_K^2} ,
\]
which gives (\ref{EK-3}). Here, $\| \cdot \|_F$ is the Frobenius norm. Moreover, from (\ref{Xdot-3}) we have
\[
\| \dot{E}_{K(0)}\|_2^2 \le \| \dot{E}_{K(0)}\|_F^2 = \sum_{i=1}^d |\dot{\V{x}}_i^K-\dot{\V{x}}_0^K|^2
\le d \| \nabla \dot{X} \|_{L^\infty(\Omega)}^2 h_K^2 ,
\]
which gives (\ref{EK-4}) (with $K(0)$ being replaced by $K$).
The inequality (\ref{EK-2}) follows from (\ref{EK-3}) and (\ref{EK-4}).
\end{proof}

\begin{rem}
\label{rem:EK-1}
As in Remark~\ref{rem:Xdot-1}, if the mesh velocity field is not smooth,  (\ref{EK-4}) and (\ref{EK-2}) can be replaced by
\begin{align}
& \| \dot{E}_{K}\|_2 \le \sqrt{2 d} \, \| \dot{X} \|_{L^\infty(\Omega)} ,
\label{EK-5}
\\
& \| \dot{E}_{K}\|_2 \|E_{K}^{-1}\|_2 \le \frac{d\, \sqrt{2}}{a_K} \| \dot{X} \|_{L^\infty(\Omega)} .
\label{EK-6}
\end{align}
\end{rem}

In FE computation, a basis function $\tilde{\psi}$ on $K(t)$ is typically defined as the composite
function of a basis function $\psi$ on $K(0)$ with the affine mapping $F_{K(t)}$, i.e.,
\beq
\tilde{\psi}(\V{y}, t) = \psi(F_{K(t)}^{-1}(\V{y})),\quad
\tilde{\psi}(F_{K(t)}(\V{x}), t) = \psi(\V{x}) .
\label{phi-1}
\eeq
Taking gradient of both sides of the second equation with respect to $\V{x}$, we get
\beq
\nablat \tilde{\psi} (F_{K(t)}(\V{x}), t) =  \mathbb{J}^{-T} \nabla \psi (\V{x}) .
\label{phi-2}
\eeq

\begin{lem}
\label{lem:phi-1}
Assume that basis functions $\tilde{\psi}$ on $K(t)$ and $\psi$ on $K(0)$ are related through (\ref{phi-1}). Then,
\begin{align}
& \dot{\tilde{\psi}}|_{t = 0} = 0,
\quad \dot{} \quad
\label{phi-3-0}
\\
& \dot{(\nablat  \tilde{\psi})}\; \; |_{t=0} 
 = - E_{K(0)}^{-T} \dot{E}_{K(0)}^T \nabla \psi .
\label{phi-3}
\end{align}
\end{lem}

\begin{proof}
The equality (\ref{phi-3-0}) follows from differentiating the second equation in (\ref{phi-1}) with respect to time.
The equality (\ref{phi-3}) can be obtained by differentiating (\ref{phi-2}) with respect to time, taking $t = 0$, and
using Lemma~\ref{lem:Jdot-1}.
\end{proof}

\begin{lem}
\label{lem:fun-1}
Consider a function $f = f(\V{x})$ defined on $K$ and let $\tilde{f} = \tilde{f}(\V{y},t) = f(F_{K(t)}(\V{x}))$.
Then,
\beq
\dot{\tilde{f}}|_{t=0} = \nabla f \cdot \dot{X}_h |_K .
\label{fun-1}
\eeq
\end{lem}

\begin{proof}
Differentiating $\tilde{f} = \tilde{f}(\V{y},t) = f(F_{K(t)}(\V{x}))$ with respect to $t$, we get
\[
\dot{\tilde{f}} = \nabla f (F_{K(t)}(\V{x})) \cdot \dot{F}_{K(t)}(\V{x}).
\]
Taking $t = 0$ and using Lemma~\ref{lem:FKdot-0}, we obtain (\ref{fun-1}).
\end{proof}

\begin{lem}
\label{lem:FE-1}
Consider an FE approximation $v_h = \sum_{i} v_i \psi_i(\V{x})$ and its perturbation
$\tilde{v}_h(\V{y},t) = \sum_{i} v_i(t) \tilde{\psi}_i(\V{y},t)$. Then,
\begin{align}
& \dot{\tilde{v}}_h |_{K, t = 0} =  \dot{v}_h |_K, \quad  \forall K \in \mathcal{T}_h
\label{der-4}
\\
& \dot{(\nablat \tilde{v}_h)}\;\; |_{K, t=0} = - E_{K}^{-T} \dot{E}_{K}^T \nabla v_h |_K + \nabla \dot{v}_h|_K,
\quad \forall K \in \mathcal{T}_h
\label{der-5}
\end{align}
where $\dot{v}_h = \sum_i \dot{v}_i(0) \psi_i(\V{x})$.
\end{lem}

\begin{proof}
Restricted on $K(t)$, using (\ref{phi-1}) we can rewrite $\tilde{v}_h$ into
\beq
\tilde{v}_h(F_{K(t)}(\V{x}),t) = \sum_{i} v_i(t) \psi_i(\V{x}) |_{K(0)}.
\label{vh-1}
\eeq
Differentiating this with respect to $t$, we get
\[
\dot{\tilde{v}}_h(F_{K(t)}(\V{x}),t) = \sum_{i} \dot{v}_i(t) \psi_i(\V{x})  |_{K(0)}.
\]
Taking $t= 0$, we get (\ref{der-4}).

Applying $\nabla \cdot$ to (\ref{vh-1}) and using the chain rule, we get
\[
\nablat \tilde{v}_h(F_{K(t)}(\V{x}),t) = \mathbb{J}^{-T} \sum_{i} v_i(t) \nabla \psi_i(\V{x}) |_{K(0)} .
\]
Differentiating this with respect to $t$, we obtain
\[
\dot{(\nablat \tilde{v}_h)} \; \; = (\dot{(\mathbb{J}^{-1})})^T \sum_{i} v_i(t) \nabla \psi_i(\V{x}) |_{K(0)}
+ \mathbb{J}^{-T} \sum_{i} \dot{v}_i(t) \nabla \psi_i(\V{x}) |_{K(0)} .
\]
Taking $t = 0$ and using Lemma~\ref{lem:Jdot-1}, we obtain (\ref{der-5}).
\end{proof}

\section{Mesh sensitivity analysis for the finite element solution}
\label{SEC:sensitivity}

In this section we analyze the mesh sensitivity for the FE solution $u_h = \sum_{i} u_i \psi_i(\V{x})$
satisfying (\ref{fem-1}).
On the deformed mesh $\mathcal{T}_h(t)$, the perturbed FE solution can be expressed
as $\tilde{u}_h = \sum_{i} u_i(t) \tilde{\psi}_i(\V{y},t)$. We assume that $\tilde{u}_h$ is differential.
From Lemma~\ref{lem:FE-1}, the material derivative of
$\tilde{u}_h$ at $t=0$ is given by $\dot{\tilde{u}}_h |_{t=0} = \sum_{i} \dot{u}_i(0) \psi_i(\V{x})$. We denote
this by $\dot{u}_h$, i.e., $\dot{u}_h = \sum_{i} \dot{u}_i(0) \psi_i(\V{x})$.
This derivative measures the change in $u_h$ with mesh deformation.
In this section, we first derive the FE formulation for $\dot{u}_h$ and then establish a bound for $\|\nabla \dot{u}_h\|_{L^2(\Omega)}$.

\begin{thm}
\label{thm:fem-1}
The material derivative $\dot{u}_h$ satisfies
\begin{align}
& \int_{\Omega} \left ( {a} \nabla \dot{u}_h \cdot \nabla \psi
+ ({\V{b}} \cdot \nabla \dot{u}_h){\psi} + {c} \dot{u}_h {\psi}\right ) d \V{x}
\nonumber
\\
& =  \sum_{K} \int_{K} \left ( {a} \nabla u_h \cdot  (\dot{E}_{K}E_{K}^{-1}+ E_{K}^{-T} \dot{E}_{K}^{T}) \cdot \nabla \psi
+ ({\V{b}} \cdot E_{K}^{-T} \dot{E}_{K}^{T} \cdot \nabla u_h){\psi} \right ) d \V{x}
\nonumber
\\
& \quad \quad - \int_{\Omega} \left ( (\nabla a \cdot \dot{X}_h)  (\nabla {u}_h \cdot \nabla {\psi})
+ a ( \nabla {u}_h \cdot \nabla {\psi}) (\nabla \cdot \dot{X}_h) \right ) d \V{x}
\nonumber
\\
& \quad \quad - \int_{\Omega} \left ( 
 (\dot{X}_h \cdot \nabla \V{b} \cdot \nabla {u}_h) {\psi} 
+ ({\V{b}} \cdot \nabla {u}_h){\psi}  (\nabla \cdot \dot{X}_h) \right ) d \V{x}
\nonumber
\\
& \quad \quad + \int_{\Omega} \left (
c \psi (\nabla u_h \cdot \dot{X}_h) +  {c} {u}_h (\nabla \psi \cdot \dot{X}_h) \right ) d \V{x}
 - \int_{\Omega} f (\nabla \psi \cdot \dot{X}_h) d \V{x} ,
 \qquad \forall \psi \in V_h .
\label{fem-4}
\end{align}
\end{thm}

\begin{proof}
We start with rewriting (\ref{fem-1}) into
\beq
\sum_{K} \int_{K} \left ( a \nabla u_h \cdot \nabla \psi  + (\V{b}\cdot \nabla u_h)\psi  + c u_h \psi\right ) d \V{x}
 = \sum_{K} \int_{K} f \psi d \V{x}, \quad \forall \psi \in V_h .
\label{fem-2}
\eeq
On the deformed mesh $\mathcal{T}_h(t)$, the perturbed FE solution can be expressed as
$\tilde{u}_h = \sum_{i} u_i(t) \tilde{\psi}_i(\V{y},t)$. Moreover, (\ref{fem-2}) becomes
\beq
\sum_{K(t)} \int_{K(t)}  \left ( \tilde{a} \nablat \tilde{u}_h \cdot \nablat \tilde{\psi}  + (\tilde{\V{b}}\cdot \nablat \tilde{u}_h)\tilde{\psi}
 + \tilde{c} \tilde{u}_h \tilde{\psi}\right ) d \V{y}
= \sum_{K(t)} \int_{K(t)} f(\V{y}) \tilde{\psi} d \V{y}, \quad \forall \psi \in V_h(t) 
\label{fem-3}
\eeq
where $V_h(t) = \text{span}\{ \tilde{\psi}_1, \tilde{\psi}_2, ... \}$.
Transforming all integrals into $K(0)$, we obtain
\begin{align*}
& \sum_{K(0)} \int_{K(0)} \left ( \tilde{a} \nablat \tilde{u}_h \cdot \nablat \tilde{\psi}
+ (\tilde{\V{b}} \cdot \nablat \tilde{u}_h)\tilde{\psi}
+ \tilde{c} \tilde{u}_h \tilde{\psi}\right ) J d \V{x}
= \sum_{K(0)} \int_{K(0)} f(F_{K(t)}(\V{x})) \tilde{\psi} J d \V{x} .
\end{align*}
Differentiating both sides of the above equation with respect to $t$ while keeping $\V{x}$ fixed, we have
\begin{align*}
& \sum_{K(0)} \int_{K(0)} \left ( \tilde{a} \dot{(\nablat \tilde{u}_h)}\;\; \cdot \nablat \tilde{\psi}
+ (\tilde{\V{b}} \cdot \dot{(\nablat \tilde{u}_h)}\;\;)\tilde{\psi}
+ \tilde{c} \dot{\tilde{u}}_h \tilde{\psi}\right ) J d \V{x}
\\
& = \quad  - \sum_{K(0)} \int_{K(0)} \left ( \tilde{a} \nablat \tilde{u}_h \cdot \dot{(\nablat \tilde{\psi})}\;\;
+ (\tilde{\V{b}} \cdot \nablat \tilde{u}_h)\dot{\tilde{\psi}}
+ \tilde{c} \tilde{u}_h \dot{\tilde{\psi}}\right ) J d \V{x}
\\
& \quad \quad - \sum_{K(0)} \int_{K(0)} \left ( \dot{\tilde{a}} \nablat \tilde{u}_h \cdot \nablat \tilde{\psi}
+ (\dot{\tilde{\V{b}}} \cdot \nablat \tilde{u}_h)\tilde{\psi}
+ \dot{\tilde{c}} \tilde{u}_h \tilde{\psi}\right ) J d \V{x}
 \\
 & \quad \quad - \sum_{K(0)} \int_{K(0)} \left ( \tilde{a} \nablat \tilde{u}_h \cdot \nablat \tilde{\psi}
+ (\tilde{\V{b}} \cdot \nablat \tilde{u}_h)\tilde{\psi}
+ \tilde{c} \tilde{u}_h \tilde{\psi}\right ) \dot{J} d \V{x}
 \\
 & \quad \quad + \sum_{K(0)} \int_{K(0)} \left ( \dot{\tilde{f}} \tilde{\psi} J + \tilde{f} \dot{\tilde{\psi}} J
 + \tilde{f} \tilde{\psi} \dot{J} \right )d \V{x} .
\end{align*}
Using Lemmas~\ref{lem:Jdot-1}, \ref{lem:phi-1}, \ref{lem:fun-1}, and \ref{lem:FE-1}, taking $t = 0$, and
noticing that $J = 1$ and $\V{y} = \V{x}$ at $t = 0$, we get 
\begin{align*}
& \sum_{K(0)} \int_{K(0)} \left ( {a} (\nabla \dot{u}_h - E_{K(0)}^{-T} \dot{E}_{K(0)}^{T}\nabla u_h) \cdot \nabla \psi
+ ({\V{b}} \cdot (\nabla \dot{u}_h - E_{K(0)}^{-T} \dot{E}_{K(0)}^{T}\nabla u_h)){\psi}
+ {c} \dot{u}_h {\psi}\right ) d \V{x}
\\
& = \quad  - \sum_{K(0)} \int_{K(0)} \left ( {a} \nabla {u}_h \cdot (- E_{K(0)}^{-T} \dot{E}_{K(0)}^{T}\nabla \psi)\right ) d \V{x}
\\
& \quad \quad - \sum_{K(0)} \int_{K(0)} \left ( (\nabla a \cdot \dot{X}_h)  \nabla {u}_h \cdot \nabla {\psi}
+ (\dot{X}_h \cdot \nabla \V{b} \cdot \nabla {u}_h) {\psi}
+ (\nabla c\cdot \dot{X}_h) {u}_h {\psi}\right ) d \V{x}
 \\
 & \quad \quad - \sum_{K(0)} \int_{K(0)} \left ( {a} \nabla {u}_h \cdot \nabla {\psi}
+ ({\V{b}} \cdot \nabla {u}_h){\psi} + {c} {u}_h {\psi}\right ) (\nabla \cdot \dot{X}_h) d \V{x}
 \\
 & \quad \quad + \sum_{K(0)} \int_{K(0)} \left ( (\nabla f \cdot \dot{X}_h) \psi 
 + f {\psi} (\nabla \cdot \dot{X}_h) \right )d \V{x} .
\end{align*}
Noticing that $K(0) = K$, we can rewrite the above equation into
\begin{align*}
& \int_{\Omega} \left ( {a} \nabla \dot{u}_h \cdot \nabla \psi
+ ({\V{b}} \cdot \nabla \dot{u}_h){\psi} + {c} \dot{u}_h {\psi}\right ) d \V{x}
\\
& =  \sum_{K} \int_{K} \left ( {a} \nabla u_h \cdot  (\dot{E}_{K}E_{K}^{-1}+ E_{K}^{-T} \dot{E}_{K}^{T}) \cdot \nabla \psi
+ ({\V{b}} \cdot E_{K}^{-T} \dot{E}_{K}^{T} \cdot \nabla u_h){\psi} \right ) d \V{x}
\\
& \quad \quad - \int_{\Omega} \left ( (\nabla a \cdot \dot{X}_h)  \nabla {u}_h \cdot \nabla {\psi}
+ (\dot{X}_h \cdot \nabla \V{b} \cdot \nabla {u}_h) {\psi}
+ (\nabla c\cdot \dot{X}_h) {u}_h {\psi}\right ) d \V{x}
 \\
 & \quad \quad - \int_{\Omega} \left ( {a} \nabla {u}_h \cdot \nabla {\psi}
+ ({\V{b}} \cdot \nabla {u}_h){\psi} + {c} {u}_h {\psi}\right ) (\nabla \cdot \dot{X}_h) d \V{x}
 \\
 & \quad \quad + \int_{\Omega} \left ( (\nabla f \cdot \dot{X}_h) \psi 
 + f {\psi} (\nabla \cdot \dot{X}_h) \right )d \V{x} .
\end{align*}
This can be rewritten into
\begin{align*}
& \int_{\Omega} \left ( {a} \nabla \dot{u}_h \cdot \nabla \psi
+ ({\V{b}} \cdot \nabla \dot{u}_h){\psi} + {c} \dot{u}_h {\psi}\right ) d \V{x}
\\
& =  \sum_{K} \int_{K} \left ( {a} \nabla u_h \cdot  (\dot{E}_{K}E_{K}^{-1}+ E_{K}^{-T} \dot{E}_{K}^{T}) \cdot \nabla \psi
+ ({\V{b}} \cdot E_{K}^{-T} \dot{E}_{K}^{T} \cdot \nabla u_h){\psi} \right ) d \V{x}
\\
& \quad \quad - \int_{\Omega} \left ( (\nabla a \cdot \dot{X}_h)  (\nabla {u}_h \cdot \nabla {\psi})
+ a ( \nabla {u}_h \cdot \nabla {\psi}) (\nabla \cdot \dot{X}_h) \right ) d \V{x}
\\
& \quad \quad - \int_{\Omega} \left ( 
 (\dot{X}_h \cdot \nabla \V{b} \cdot \nabla {u}_h) {\psi} 
+ ({\V{b}} \cdot \nabla {u}_h){\psi}  (\nabla \cdot \dot{X}_h) \right ) d \V{x}
\\
& \quad \quad - \int_{\Omega} \left ( 
(\nabla c\cdot \dot{X}_h) {u}_h {\psi} +  {c} {u}_h {\psi} (\nabla \cdot \dot{X}_h) \right ) d \V{x}
 \\
 & \quad \quad + \int_{\Omega} \left ( (\nabla f \cdot \dot{X}_h) \psi 
 + f {\psi} (\nabla \cdot \dot{X}_h) \right )d \V{x} .
\end{align*}
Using the divergence theorem, we get
\begin{align}
& \int_{\Omega} \left ( {a} \nabla \dot{u}_h \cdot \nabla \psi
+ ({\V{b}} \cdot \nabla \dot{u}_h){\psi} + {c} \dot{u}_h {\psi}\right ) d \V{x}
\nonumber
\\
& =  \sum_{K} \int_{K} \left ( {a} \nabla u_h \cdot  (\dot{E}_{K}E_{K}^{-1}+ E_{K}^{-T} \dot{E}_{K}^{T}) \cdot \nabla \psi
+ ({\V{b}} \cdot E_{K}^{-T} \dot{E}_{K}^{T} \cdot \nabla u_h){\psi} \right ) d \V{x}
\nonumber
\\
& \quad \quad - \int_{\Omega} \left ( (\nabla a \cdot \dot{X}_h)  (\nabla {u}_h \cdot \nabla {\psi})
+ a ( \nabla {u}_h \cdot \nabla {\psi}) (\nabla \cdot \dot{X}_h) \right ) d \V{x}
\nonumber
\\
& \quad \quad - \int_{\Omega} \left ( 
 (\dot{X}_h \cdot \nabla \V{b} \cdot \nabla {u}_h) {\psi} 
+ ({\V{b}} \cdot \nabla {u}_h){\psi}  (\nabla \cdot \dot{X}_h) \right ) d \V{x}
\nonumber
\\
& \quad \quad - \int_{\partial \Omega} c u_h \psi \dot{X}_h\cdot \V{n} d S + \int_{\Omega} \left (
c \psi (\nabla u_h \cdot \dot{X}_h) +  {c} {u}_h (\nabla \psi \cdot \dot{X}_h) \right ) d \V{x}
\nonumber
 \\
 & \quad \quad + \int_{\partial \Omega} f \psi \dot{X}_h\cdot \V{n} d S - \int_{\Omega} f (\nabla \psi \cdot \dot{X}_h) d \V{x} ,
 \label{fem-7}
\end{align}
where $\V{n}$ is the outward unit normal of $\partial \Omega$. We obtain (\ref{fem-4}) by noticing that the surface integrals vanish
since $\psi = 0$ on $\partial \Omega$.
\end{proof}

\begin{thm}
\label{thm:fem-2}
Assume that $\mathcal{T}_h$ is a simplicial mesh with the minimum element height $\min_K a_K > 0$. Then, the material derivative
of the FE solution to BVP (\ref{bvp-1}) and (\ref{bc-1}) due to mesh deformation is bounded by
\begin{align}
a_0 \| \nabla \dot{u}_h \|_{L^2(\Omega)} & \le \| f \|_{L^2(\Omega)}\left ( 1 + C_{\Omega} \|\nabla a \|_{L^\infty(\Omega)}
+ C_{\Omega}^2 \|\nabla \V{b} \|_{L^\infty(\Omega)}  + 2 C_{\Omega}^2 \| c \|_{L^\infty(\Omega)} \right ) \| \dot{X} \|_{L^\infty(\Omega)}
\nonumber
\\
& \quad  +  \| f \|_{L^2(\Omega)}\left ( 3 d C_{\Omega} \| a \|_{L^\infty(\Omega)} + 2 d C_{\Omega}^2 \| \V{b} \|_{L^\infty(\Omega)}\right )
 \| \nabla \dot{X} \|_{L^\infty(\Omega)} \max_K \frac{h_K}{a_K} ,
\label{fem-5}
\end{align}
where $C_{\Omega}$ is the constant appearing in Poincar\'{e}'s inequality that depends only on $\Omega$.
\end{thm}

\begin{proof}
Recall that $\dot{u}_h \equiv \sum_i \dot{u}_i(0) \psi_i$ belongs to $V_h$. Taking $\psi = \dot{u}_h$ in (\ref{fem-4})
and using the Cauchy-Schwarz inequality, Poincar\'{e}'s inequality, and (\ref{bvp-3}), we get
\begin{align*}
a_0 \| \nabla \dot{u}_h \|_{L^2(\Omega)}
& \le \| a \|_{L^\infty(\Omega)} \| \nabla u_h\|_{L^2(\Omega)} \left ( \| \nabla \cdot \dot{X} \|_{L^\infty(\Omega)}
+ 2 \max_K \| \dot{E}_K E_K^{-1} \|_2 \right )
\\
& \quad + \| \nabla a \|_{L^\infty(\Omega)} \| \nabla u_h\|_{L^2(\Omega)} \| \dot{X}_h \|_{L^\infty(\Omega)}
\\
& \quad + C_{\Omega} \| \V{b} \|_{L^\infty(\Omega)} \| \nabla u_h\|_{L^2(\Omega)} \left ( \| \nabla \cdot \dot{X} \|_{L^\infty(\Omega)}
+ \max_K \| \dot{E}_K E_K^{-1} \|_2 \right )
\\
& \quad + C_{\Omega} \| \nabla \V{b} \|_{L^\infty(\Omega)} \| \nabla u_h\|_{L^2(\Omega)} \| \dot{X}_h \|_{L^\infty(\Omega)}
+ 2 C_{\Omega} \| c \|_{L^\infty(\Omega)} \| \nabla u_h\|_{L^2(\Omega)} \| \dot{X}_h \|_{L^\infty(\Omega)}
\\
& \quad + \| f \|_{L^2(\Omega)}  \| \dot{X}_h \|_{L^\infty(\Omega)} .
\end{align*}
Taking $\psi = u_h$ in (\ref{fem-1}) and using the Cauchy-Schwarz inequality and Poincar\'{e}'s inequality, we get
\[
\| \nabla u_h \|_{L^2(\Omega)} \le C_{\Omega} \| f \|_{L^2(\Omega)} .
\]
Combining the above results and using Lemmas~\ref{lem:Xh-1} and \ref{lem:EK-1}, we obtain (\ref{fem-5}).
\end{proof}

We emphasize that the bound in (\ref{fem-5}) has been obtained without assuming that the mesh is regular. In fact,
the mesh can be arbitrary, isotropic or anisotropic, uniform or nonuniform, as long as it is simplicial and
has a positive minimum element height.
For meshes with large aspect ratio, the factor $\max_K h_K/a_K$ is large and the bound in (\ref{fem-5})
is more sensitive to $\| \nabla \dot{X}\|_{L^2(\Omega)}$.
Moreover, the bound shows that the size and gradient of the mesh velocity field can have effects
on the FE solution. The former is insensitive to the shape of mesh elements whereas the effects from the gradient
of the mesh velocity are proportional to the maximum element aspect ratio $\max_K h_K/a_K$.
Furthermore, (\ref{fem-5}) is homogeneous about time derivatives. Thus, $\dot{X}$ can be viewed as mesh displacement instead of mesh velocity. The bound shows that the change in the FE solution is small when $\dot{X}$ and $\nabla \dot{X}$ are small, implying
a continuous dependence of the FE solution on the mesh.

\begin{rem}
\label{rem:fem-1}
If the mesh velocity field is not smooth, from Remarks~\ref{rem:Xdot-1} and (\ref{rem:EK-1}) and inequality (\ref{fem-7})
we can see that (\ref{fem-5}) can be replaced by
\begin{align}
a_0 \| \nabla \dot{u}_h \|_{L^2(\Omega)} & \le \| f \|_{L^2(\Omega)}\left ( 1 + C_{\Omega} \|\nabla a \|_{L^\infty(\Omega)}
+ C_{\Omega}^2 \|\nabla \V{b} \|_{L^\infty(\Omega)}  + 2 C_{\Omega}^2 \| c \|_{L^\infty(\Omega)} \right ) \| \dot{X} \|_{L^\infty(\Omega)}
\nonumber
\\
& \quad  +  \| f \|_{L^2(\Omega)}\left ( 3 d C_{\Omega} \| a \|_{L^\infty(\Omega)} + 2 d C_{\Omega}^2 \| \V{b} \|_{L^\infty(\Omega)}\right )
 \| \dot{X} \|_{L^\infty(\Omega)} \frac{1}{\min\limits_K a_K} .
\label{fem-6}
\end{align}
For a given mesh, $\min\limits_K a_K$ is fixed. Thus, (\ref{fem-6}) shows that the FE solution depends continuously
on the mesh even in the situation where the mesh velocity field is not smooth.
\end{rem}

\section{A numerical example}
\label{SEC:numerics}

In this section we present some numerical results for a two-dimensional example in the form (\ref{bvp-1}) where
$\Omega = (0,1)\times (0,1)$, $a = 1$, $\V{b} = (1,2)^T$, $c=0$, and $f$ is chosen such that the exact solution
of the BVP is given by $u = \sin(2\pi x) \sin(3 \pi y)$. We use linear finite elements and choose the base mesh to
be a triangular mesh by first partitioning $\Omega$ into $N\times N$ small rectangles and then dividing each small rectangle
into four triangles using its diagonal lines. We first consider a smooth velocity field,
\[
\dot{X} = \sin(\pi x) \sin(2 \pi y)
\]
and deform the mesh through (\ref{xt-1}). The results are listed in Table~\ref{tab:smooth}. We can see that
$\| \nabla \dot{u}_h\|_{L^2(\Omega)}$ is a linear function of $t$. Moreover, the error is almost independent of
the mesh size. This is consistent with (\ref{fem-5}) since the only mesh-dependent factor $\max_K h_K/a_K$ is
constant when the mesh is being refined for the current situation.

Next we consider a random mesh velocity field. In this case, we generate $\dot{X}_i = \dot{X}(\V{x}_i)$ ($i = 1, 2, ...$)
using a uniformly distributed pseudorandom number generator and scale them to the range $(-1,1)$. Then we perturb
the mesh through (\ref{xt-1}). In this case, the velocity field is not smooth. The results are listed in Table~\ref{tab:random}
where $\| \nabla \dot{u}_h\|_{L^2(\Omega)}$ is shown as the average of the corresponding values obtained
with twenty repeated runs for each pair of $t$ and the mesh size. From the table we can see that $\| \nabla \dot{u}_h\|_{L^2(\Omega)}$
is still linear about $t$. Interestingly, $\| \nabla \dot{u}_h\|_{L^2(\Omega)}$ is mesh-dependent: its values are about twice larger
for the $80\times 80$ mesh than those for the $40\times 40$ mesh. This can be explained using (\ref{fem-6}) which contains
a mesh-dependent factor $1/(\min_K a_K)$. This factor is about twice larger for the $80\times 80$ mesh than
for the $40\times 40$ mesh. Moreover, we can see that $\| \nabla \dot{u}_h\|_{L^2(\Omega)}$
in Table~\ref{tab:random} are larger than those
in Table~\ref{tab:smooth}, which reflects the nature of the bounds in (\ref{fem-5}) and (\ref{fem-6}) for smooth
and nonsmooth velocity fields.

\begin{table}[htb]
\begin{center}
\caption{$\| \nabla \dot{u}_h\|_{L^2(\Omega)}$ for smooth velocity field.}
\vspace{2mm}
\begin{tabular}{|c|c|c|c|c|c|}\hline \hline
Mesh size & $t =$ 1e-6 & $t =$ 1e-5 & $t =$ 1e-4 & $t =$ 1e-3 & $t =$ 1e-2 \\ \hline
$40\times 40$ & 3.759e-5 & 3.759e-4 & 3.759e-3 & 3.759e-2 & 3.760e-1 \\
$80\times 80$ & 3.770e-5 & 3.770e-4 & 3.770e-3 & 3.771e-2 & 3.772e-1 \\
\hline \hline
\end{tabular}
\label{tab:smooth}
\end{center}
\end{table}

\begin{table}[htb]
\begin{center}
\caption{$\| \nabla \dot{u}_h\|_{L^2(\Omega)}$ (average) for random velocity field. The shown values are obtained
as the average of the corresponding values obtained
with twenty repeated runs for each pair of $t$ and the mesh size.}
\vspace{2mm}
\begin{tabular}{|c|c|c|c|c|}\hline \hline
Mesh size & $t =$ 1e-6 & $t =$ 1e-5 & $t =$ 1e-4 & $t =$ 1e-3 \\ \hline
$40\times 40$ & 3.641e-4 & 3.640e-3 & 3.633e-2 & 3.656e-1 \\
$80\times 80$ & 7.367e-4 & 7.341e-3 & 7.353e-2 & 7.432e-1 \\
\hline \hline
\end{tabular}
\label{tab:random}
\end{center}
\end{table}

\section{Conclusions}
\label{SEC:conclusions}

We have presented an analysis on the mesh sensitivity for the finite element solution of
a boundary value problem of linear elliptic partial differential equations. The main result is stated in
Theorem~\ref{thm:fem-2} where a bound is obtained for the change in the finite element solution
in terms of the mesh deformation and its gradient. A similar bound is given in (\ref{fem-6}) when the mesh velocity field
is nonsmooth. These results show how the finite element solution depends continuously on the mesh.
The results have been obtained in any dimension and for arbitrary unstructured simplicial
meshes, general linear elliptic partial differential equations, and general finite element approximations.


\end{document}